\documentclass[a4paper]{amsart}
\usepackage{graphicx}
\usepackage{amsmath}
\usepackage{amssymb}
\usepackage{amsfonts}
\usepackage{amsthm}
\usepackage[all]{xy}
\usepackage{eucal}
\newcommand{\A}{\mathbf{A}}
\newcommand{\NA}{N\!\mathbf{A}}
\newcommand{\sSets}{\mathbf{sSets}}

\theoremstyle{plain}
\newtheorem{theorem}{Theorem}[section]
\newtheorem{theorema}{Theorem}

\newtheorem{lemma}[theorem]{Lemma}
\newtheorem{proposition}[theorem]{Proposition}
\newtheorem{propositiona}[theorema]{Proposition}

\newtheorem*{proposition*}{Proposition}
\newtheorem*{corollary*}{Corollary}

\theoremstyle{definition}

\newtheorem{remark}[theorem]{Remark}

\makeatletter
\@addtoreset{theorem}{section}
\@addtoreset{subsection}{section}
\makeatother

\begin{document}

\title{Left fibrations and homotopy colimits II}

\author[Gijs Heuts]{Gijs Heuts}
\address{University of Copenhagen, Department of Mathematical Sciences, Universitetsparken 5, DK-2100 Copenhagen, Denmark}
\email{gheuts@math.ku.dk}

\author[Ieke Moerdijk]{Ieke Moerdijk}
\address{Radboud Universiteit Nijmegen, Institute for Mathematics, Astrophysics and Particle Physics, Heyendaalseweg 135, 6525 AJ Nijmegen, The Netherlands}
\email{i.moerdijk@math.ru.nl}

\date{}

\begin{abstract}
For a small simplicial category $\A$, we prove that the homotopy colimit functor from the category of simplicial diagrams on $\A$ to the category of simplicial sets over the homotopy-coherent nerve of $\A$ provides a left Quillen equivalence between the projective model structure on the former category and the covariant model structure on the latter. We compare this Quillen equivalence to the straightening-unstraightening equivalence previously established by Lurie, where the left adjoint goes in the opposite direction. The existence of left Quillen functors in both directions considerably simplifies the proof that these constructions provide Quillen equivalences. The results of this paper generalize those of part I, where $\A$ was an ordinary category. The proofs for a simplicial category are more involved and can be read independently.
\end{abstract}

\maketitle

%\tableofcontents

\section{Introduction and main results}

For the category $\mathbf{sSets}$ of simplicial sets and the category $\mathbf{sCat}$ of simplicially enriched categories (usually more briefly referred to as \emph{simplicial categories}), there is an adjoint pair
\[
\xymatrix{
w_!: \mathbf{sSets} \ar@<.5ex>[r] & \mathbf{sCat}: w^* \ar@<.5ex>[l]
}
\]
which is known to be a Quillen equivalence for the Joyal model structure on the former and the Bergner model structure on the latter category \cite{cisinskimoerdijk3, joyalcomparison, htt}. We will recall the definition of this pair in Section \ref{sec:review}; the functor $w_!$ is denoted $\mathfrak{C}$ in \cite{htt}, where the functor $w^*$ is written $N$. The latter functor is usually referred to as the \emph{homotopy-coherent nerve}. Our notation is inspired by the $W$-construction of Boardman and Vogt \cite{boardmanvogt}. 

For a simplicial category $\mathbf{A}$, we will define a functor 
\begin{equation*}
h_!: \mathbf{sSets}^{\mathbf{A}} \longrightarrow \mathbf{sSets}/w^*\mathbf{A}.
\end{equation*}
Here $\mathbf{sSets}^{\mathbf{A}}$ denotes the category of simplicially enriched functors from $\mathbf{A}$ to $\mathbf{sSets}$ and $\mathbf{sSets}/w^*\mathbf{A}$ the slice category of simplicial sets over $w^*\mathbf{A}$. We will consider the projective model structure on the first and the covariant model structure on the second category, which we recall in Section \ref{sec:review}. The functor $h_!$ will be a simplicial left Kan extension of a simplicial functor
\begin{equation*}
h: \mathbf{A}^{\mathrm{op}} \longrightarrow \mathbf{sSets}/w^*\mathbf{A}.
\end{equation*}
On objects it satisfies
\begin{equation*}
h(a) = a/w^*\mathbf{A},
\end{equation*}
where for a general simplicial set $X$ and vertex $a \in X_0$, the simplicial set $a/X$ has as its $n$-simplices the $(n+1)$-simplices of $X$ whose 0'th vertex is $a$, with the evident simplicial structure. Our first result is the following:

\begin{propositiona}
\label{prop:h!}
The functor $h_!$ is part of an adjoint pair
\[
\xymatrix{
h_!: \mathbf{sSets}^{\mathbf{A}} \ar@<.5ex>[r] & \mathbf{sSets}/w^*\mathbf{A}: h^* \ar@<.5ex>[l]
}
\]
which is in fact a Quillen adjunction. Furthermore, the composition of $h_!$ with the forgetful functor $\mathbf{sSets}/w^*\mathbf{A} \rightarrow \mathbf{sSets}$ is equivalent to the homotopy colimit functor
\begin{equation*}
\mathrm{hocolim}_{\mathbf{A}}: \mathbf{sSets}^{\mathbf{A}} \longrightarrow \mathbf{sSets}.
\end{equation*}
\end{propositiona}

\begin{remark}
In part I of this paper \cite{leftfibrations} we defined a functor $h_!: \mathbf{sSets}^{\mathbf{A}} \rightarrow \mathbf{sSets}/\NA$ for an ordinary category $\mathbf{A}$ and its nerve $\NA$. For a discrete simplicial category $\mathbf{A}$, viewed as an ordinary category, the functor of Proposition \ref{prop:h!} coincides with that functor after identifying $\NA$ and $w^*\mathbf{A}$, which are indeed the same for \emph{discrete} simplicial categories.
\end{remark}

For a simplicial set $X$ there is an adjoint pair of functors
\[
\xymatrix{
r_!: \mathbf{sSets}/X \ar@<.5ex>[r] & \mathbf{sSets}^{w_! X}: r^*. \ar@<.5ex>[l]
}
\]
This is the pair which Lurie calls the \emph{straightening} and \emph{unstraightening} functors \cite{htt}. It can be described as follows. For a map of simplicial sets $p: Y \rightarrow X$, consider a pushout square
\[
\xymatrix{
Y \ar[d]_p \ar[r] & Y^\triangleleft \ar[d] \\
X \ar[r]_-i & Y(p).
}
\]
Here $Y^{\triangleleft}$ denotes the \emph{left cone} on $Y$, defined as the join $\Delta^0 \star Y$. If $-\infty$ denotes the image in $Y(p)$ of the cone vertex of this $Y^{\triangleleft}$, then $r_!(p)$ is defined by the formula
\begin{equation*}
r_!(p): w_! X \longrightarrow \mathbf{sSets}: x \longmapsto w_!Y(p)(-\infty, i(x)).
\end{equation*}
An equivalent description we will sometimes use is the following. For $x$ a vertex of $X$, and thus an object of $w_!X$, we have
\begin{equation*}
r_!(Y \rightarrow X)(x) = (w_!p)^\ast \bigl(w_!X(-, x)\bigr) \otimes_{w_!Y} w_!Y^{\triangleleft}(-\infty, -).
\end{equation*}
Here $(w_!p)^\ast \bigl(w_!X(-, x)\bigr)$ is interpreted as a right module over the simplicial category $w_!Y$ (i.e., a contravariant simplicial functor from it to $\mathbf{sSets}$), whereas $w_!Y^{\triangleleft}(-\infty,-)$ is a left module over $w_!Y$ (i.e., a covariant such functor). The following is proved in \cite{htt}; we will include a proof in Section \ref{sec:Qpairs}.

\begin{propositiona}
\label{prop:r!}
The adjoint pair $(r_!, r^*)$ is a Quillen adjunction for the covariant model structure on $\mathbf{sSets}/X$ and the projective model structure on $\mathbf{sSets}^{w_!X}$.
\end{propositiona}

\begin{remark}
For the length of this remark, write $\tilde{r}_!$ for the functor defined above. In part I of this paper we defined a left Quillen functor $r_!: \mathbf{sSets}/\NA \rightarrow \mathbf{sSets}^{\mathbf{A}}$ for $\mathbf{A}$ an ordinary category. Write $\varepsilon: w_! \NA \rightarrow \mathbf{A}$ for the counit of the adjunction $(w_!, w^*)$ evaluated on $\mathbf{A}$ (again using that here $w^*\mathbf{A} = \NA$). Using the techniques of Section \ref{sec:cubes} it is not difficult to define a weak equivalence $\varepsilon_! \circ \tilde{r}_! \rightarrow r_!$, so that our `new' $r_!$ is very much analogous to the one considered in part I.
\end{remark}

We will exploit the existence of both these Quillen pairs to prove our main result:

\begin{theorema}
\label{thm:main}
The Quillen pair $(r_!, r^*)$ is a Quillen equivalence. If $\mathbf{A}$ is a fibrant simplicial category, then the pair $(h_!,h^*)$ is a Quillen equivalence as well. \end{theorema}

We will prove the theorem by showing that the functors $\mathbf{L}h_!$ and $\mathbf{L}r_!$ are mutually inverse in an appropriate sense. The crucial steps are Propositions \ref{prop:h!leftinv} and \ref{prop:h!rightinv} below. In Section \ref{sec:proof} we will explain in more detail how Theorem \ref{thm:main} can be deduced from these.

\begin{propositiona}
\label{prop:h!leftinv}
For a fibrant simplicial category $\A$, the two functors
\begin{eqnarray*}
\mathbf{L}r_! \circ \mathbf{L}h_!: \mathbf{Ho}(\mathbf{sSets}^{\A}) & \longrightarrow & \mathbf{Ho}(\mathbf{sSets}^{w_!w^*\A}) \\
\mathbf{R}\varepsilon^*: \mathbf{Ho}(\mathbf{sSets}^{\A}) & \longrightarrow & \mathbf{Ho}(\mathbf{sSets}^{w_!w^*\A})
\end{eqnarray*}
are naturally isomorphic.
\end{propositiona}

The functor $\varepsilon: w_!w^*\A \rightarrow \A$ is an equivalence of simplicial categories (since $(w_!, w^*)$ is a Quillen equivalence), so that $\mathbf{R}\varepsilon^*$ is part of a Quillen equivalence (see Proposition \ref{prop:invprojective}). Therefore the previous proposition shows that $\mathbf{L}h_!$ admits a left (quasi-)inverse. We will use the following to deduce that it also admits a right inverse:

\begin{propositiona}
\label{prop:h!rightinv}
Let $X$ be a simplicial set and write $\eta: X \rightarrow w^*w_!X$ for the unit of the adjoint pair $(w_!, w^*)$. Then the two functors
\begin{eqnarray*}
\mathbf{L}h_! \circ \mathbf{L}r_!: \mathbf{Ho}(\sSets/X) & \longrightarrow & \mathbf{Ho}(\sSets/w^*w_!X) \\
\mathbf{L}\eta_!: \mathbf{Ho}(\sSets/X) & \longrightarrow & \mathbf{Ho}(\sSets/w^*w_!X)
\end{eqnarray*}
are naturally isomorphic.
\end{propositiona}

The reader should note that the previous proposition does not quite prove that $h_!$ admits a right inverse: indeed, the functor $\mathbf{L}\eta_!$ need not be an equivalence, because $\eta: X \rightarrow w^*w_!X$ is not necessarily an equivalence in the Joyal model structure. Rather, one should first take a fibrant replacement of $w_!X$ before applying $w^*$. We remedy this defect in Section \ref{sec:proof} by examining the properties of $h_!$ with respect to base change.

\section{Cubes and the Boardman-Vogt $W$-resolution}
\label{sec:cubes}

In this section we study some cosimplicial objects which play a central role in the rest of this paper. The reader might wish to skip this section and refer back to it as needed. As usual, write $[n]$ for the linearly ordered set $\{0, \ldots, n\}$. Also, write $P^n$ for the partially ordered set of nonempty subsets of $[n]$ and $NP^n$ for its nerve, which may be thought of as a simplicial subset of the cube $\prod_{i=0}^n \Delta^1$. (In fact, $NP^n$ is also the barycentric subdivision of $\Delta^n$.) Then $NP^\bullet$ is a cosimplicial simplicial set in a natural way. Indeed, for a map $\alpha: [m] \rightarrow [n]$ in the simplex category $\mathbf{\Delta}$, define $\alpha_*: NP^m \rightarrow NP^n$ by sending a nonempty subset $S \subseteq [m]$ to $\alpha(S) \subseteq [n]$. Define another simplicial set $C^n$ as the quotient of $NP^n$ which collapses, for each $i$, the simplicial subset $\prod_{j=0}^{i-1} \Delta^1 \times \{1\} \times \prod_{j=i+1}^n \Delta^1$ onto $\prod_{j=0}^{i-1} \Delta^1 \times \prod_{j=i}^n \{1\}$. One easily verifies that the structure maps of the cosimplicial object $NP^\bullet$ also induce a cosimplicial structure on $C^\bullet$. Finally, there is a natural map of cosimplicial objects $\bar{\mu}: (NP^\bullet)^{\mathrm{op}} \rightarrow \Delta^\bullet$, defined by sending a nonempty subset $S \subset [n]$ to its minimal element. This map evidently factors through a map $\mu: (C^\bullet)^{\mathrm{op}} \rightarrow \Delta^\bullet$. Here, for a simplicial set $X$, the simplicial set $X^{\mathrm{op}}$ is its \emph{opposite} or \emph{reverse}. It is obtained from $X: \mathbf{\Delta}^{\mathrm{op}} \rightarrow \mathbf{Sets}$ by precomposing with the automorphism of $\mathbf{\Delta}^{\mathrm{op}}$ which reverses linear orderings. This automorphism is the identity on objects, but for example swaps the face maps $\partial_i: [n-1] \rightarrow [n]$ and $\partial_{n-i}: [n-1] \rightarrow [n]$.

Perhaps it will help the reader's intuition to consider the geometric realizations of the constructions described above, which we do now. The space $|NP^n|$ has as its points sequences of real numbers $(t_0, \ldots, t_n)$, where $0 \leq t_i \leq 1$ for every $i$ and at least one of the $t_i$ equals $1$. Two such sequences $(t_0, \ldots, t_n)$ and $(t'_0, \ldots, t'_n)$ are identified in $|C^n|$ if $t_i = t'_i = 1$ for some $i$ and $t_j = t'_j$ for $j < i$. The cosimplicial structure maps $\partial_k: |C^{n-1}| \rightarrow |C^n|$ and $\sigma_l: |C^n| \rightarrow |C^{n-1}|$ for $k = 0, \ldots, n$ and $l = 0, \ldots, n-1$ are described by
\begin{eqnarray*}
\partial_k(t_0, \ldots, t_{n-1}) & = & (t_0, \ldots, t_{k-1}, 0, t_k, \ldots, t_{n-1}), \\
\sigma_l(t_0, \ldots, t_n) & = & (t_0, \ldots, t_l \vee t_{l+1}, \ldots, t_n).
\end{eqnarray*}
Here $\vee$ denotes taking the supremum of two numbers. Note that in the particular case $\partial_n: |C^{n-1}| \rightarrow |C^{n}|$ it does not matter whether we write $\partial_n(t_0, \ldots, t_{n-1}) = (t_0, \ldots, t_{n-1}, 0)$ or $\partial_n(t_0, \ldots, t_{n-1}) = (t_0, \ldots, t_{n-1}, 1)$; indeed, these sequences are identified because there always exists $0 \leq i < n$ such that $t_i = 1$. Also, each point of $|C^n|$ has a unique representative of the form $(t_0, \ldots, t_{i-1}, 1, \ldots, 1)$, where $t_j$ is strictly smaller than $1$ for $j < i$. To describe the map from $|C^n|$ to the topological $n$-simplex $|\Delta^n|$ we first introduce a convenient model for the latter. Write $D^n$ for the space of sequences of real numbers $(s_1, \ldots, s_n)$ such that $0 \leq s_1 \leq \cdots \leq s_n \leq 1$. Then $D^n$ is homeomorphic to the more usual space
\begin{equation*}
|\Delta^n| = \{(t_0, \ldots, t_n) \, | \, t_i \geq 0, \Sigma_i t_i = 1\} 
\end{equation*}
via the map $\varphi: |\Delta^n| \rightarrow D^n$ sending $(t_0, \ldots, t_n)$ to the sequence $(s_1, \ldots, s_n)$ where
\begin{equation*}
s_i = \sum_{0 \leq k < i} t_k.
\end{equation*}
The natural map $|\mu|: |C^n| \rightarrow D^n$ is given by sending (the equivalence class of) a sequence $(t_0, \ldots, t_n)$ to $(s_1, \ldots, s_n)$ with
\begin{equation*}
s_k := t_0 \vee \cdots \vee t_{k-1}.
\end{equation*}
Note that there is no need to take `reverses' into account here, since the geometric realization of a simplicial set and its opposite are naturally homeomorphic.
\\

The reason we are led to consider the cosimplicial simplicial set $C^\bullet$ described above stems from the Boardman-Vogt $W$-resolution, to which we will devote the remainder of this section. This resolution provides for each simplex $\Delta^n$ a simplicial category $W\Delta^n$, naturally in $[n]$. By left Kan extension this determines an adjoint pair of functors
\[
\xymatrix{
w_!: \mathbf{sSets} \ar@<.5ex>[r] & \mathbf{sCat}: w^*, \ar@<.5ex>[l]
}
\]
already mentioned at the start of this paper. To ease the exposition we will first give a description of the topological category $|W\Delta^n|$ obtained by taking geometric realizations of all the mapping objects of $W\Delta^n$. The objects of this category are the vertices of $\Delta^n$. The mapping space $|W\Delta^n(i,j)|$ is empty for $i > j$, whereas for $0 \leq i \leq j \leq n$, the mapping space $|W\Delta^n(i,j)|$ is the space of sequences of real numbers $(1, t_{i+1}, \ldots, t_{j-1}, 1)$, where each $t_k$ is contained in the unit interval. It is convenient to think of the number $t_k$ as being a label on the object $k$. Composition in this category is given by amalgamating sequences; to be precise, composing sequences 
\begin{equation*}
(1, t_{i+1}, \ldots, t_{j-1}, 1) \in |W\Delta^n(i,j)| \quad \text{and} \quad (1, t_{j+1}, \ldots, t_{k-1},1) \in  |W\Delta^n(j,k)|
\end{equation*}
gives the sequence
\begin{equation*}
(1, t_{i+1}, \ldots, t_{j-1}, 1, t_{j+1}, \ldots, t_{k-1}, 1) \in |W\Delta^n(i,k)|.
\end{equation*}
There is an evident functor $|W\Delta^n| \rightarrow [n]$ which is the identity on objects. Conversely, there is a functor $[n] \rightarrow |W\Delta^n|$ which is again the identity on objects and, for $0 \leq i \leq j \leq n$, sends the unique map $i \rightarrow j$ to the sequence all of whose entries are 1. We should explain how the construction of $|W\Delta^n|$ is natural in $[n]$. For any map $\alpha: [m] \rightarrow [n]$ in the simplex category, the functor $\alpha_*: |W\Delta^m| \rightarrow |W\Delta^n|$ is given by $\alpha$ itself on objects, whereas on morphisms it sends a sequence $(1, t_{i+1}, \ldots, t_{j-1}, 1)$ as above to the sequence $(1, s_{\alpha(i)+1}, \ldots, s_{\alpha(j)-1},1)$ where
\begin{equation*}
s_k = \bigvee_{\alpha(l) = k} t_l \quad\quad \text{for } \alpha(i) < k < \alpha(j).
\end{equation*}
Here $\vee$ again indicates taking the supremum of a set of real numbers, with the convention that this supremum is zero if the set is empty. The purpose of the category $|W\Delta^n|$ is to describe \emph{homotopy-coherent diagrams} on $\Delta^n$. For example, for $n=2$, a continuous functor $F$ from $|W\Delta^n|$ to a topological category $\mathbf{C}$ consists of three morphisms $F(f): F(0) \rightarrow F(1)$, $F(g): F(1) \rightarrow F(2)$, $F(h): F(0) \rightarrow F(2)$ in $\mathbf{C}$ together with a homotopy (rather than an identity) between $F(h)$ and $F(g) \circ F(f)$.

To describe the simplicial category $W\Delta^n$ itself, we note that with minor modifications the construction above makes sense for any reasonable \emph{interval object} (as in \cite{bergermoerdijkBV}). For our purposes, an interval object is a simplicial set $I$ with two distinguished vertices $-$ and $+$, together with an associative operation $\vee: I \times I \rightarrow I$ for which $-$ is neutral and $+$ is absorbing. In terms of equations this means
\begin{equation*}
- \vee x = x \vee - = x \quad\quad \text{and} \quad\quad + \vee x = x \vee + = +.
\end{equation*}
In the topological setting above the relevant interval object is the real unit interval with supremum as the operation $\vee$. In the case of simplicial sets the example relevant to us here has $I$ equal to the 1-simplex $\Delta^1$ and $- = 1$, $+ = 0$, where the operation $\vee$ is now (on vertices) given by taking the \emph{minimum} of a pair $(i,j)$. The reason for this convention, rather than the obvious one with $- = 0$ and $+ = 1$, is the contravariance of the map $\mu: (NP^\bullet)^{\mathrm{op}} \rightarrow \Delta^\bullet$ described at the start of this section. 

The simplicial category $W\Delta^n$ has the vertices of $\Delta^n$ as its objects; the mapping objects in this simplicial category are defined by
\begin{equation*}
W\Delta^n(i,j) = \prod_{i < k < j} I
\end{equation*}
for $i \leq j$ and are empty for $i > j$. In other words, $W\Delta^n(i,i)$ is simply $\Delta^0$ for each $i$, whereas for $0 \leq i < j \leq n$ we have $W\Delta^n(i,j) \simeq I^{\times (j-i-1)}$. Composition $W\Delta^n(i,j) \times W\Delta^n(j,k) \rightarrow W\Delta^n(i,k)$ is now described by the inclusion
\begin{equation*}
\prod_{i < l < j} I \times \prod_{j < l < k} I  \longrightarrow \prod_{i < l < k} I
\end{equation*}
which inserts $+$ in the $j$'th factor. Naturality in $[n]$ is defined analogously to the topological case, now using the operation $\vee$ on $I$. To be explicit, for a map $\alpha: [m] \rightarrow [n]$ and $0 \leq i \leq j \leq m$ we should provide a map
\begin{equation*}
\prod_{i < k < j} I \longrightarrow \prod_{\alpha(i) < l < \alpha(j)} I.
\end{equation*}
This map itself splits into a product of factors, one for each $l$, of the form
\begin{equation*}
\prod_{i < k < j, \, \alpha(k) = l} I \longrightarrow I.
\end{equation*}
If the indexing set on the right is empty, this map is the inclusion $-: \Delta^0 \rightarrow I$. If it is not, then one repeatedly applies the operation $\vee$.
\\

To conclude this section, let us explain the relation between the Boardman-Vogt $W$-construction and the cosimplicial object $C^\bullet$. Consider the simplicial set $Q^n$ defined by the pushout diagram
\[
\xymatrix{
\Delta^n \ar[r]^-{\partial_0} \ar[d] & \Delta^{n+1} \ar[d] \\
\Delta^0 \ar[r] & Q^n.
}
\]
The top horizontal map is perhaps more naturally thought of as the inclusion $\Delta^n \rightarrow \Delta^0 \star \Delta^n$ and this identification makes it clear that $Q^\bullet$ inherits the structure of a cosimplicial object from $\Delta^\bullet$. Write $-\infty$ for the vertex of $Q^n$ which is the image of $0$ in $\Delta^{n+1}$ and $v$ for the vertex which is the image of all other vertices of $\Delta^{n+1}$ (cf. the first section). Then we have an identification of cosimplicial objects
\begin{equation*}
w_!Q^\bullet(-\infty,v) \simeq (C^\bullet)^{\mathrm{op}}.
\end{equation*}
In other words, we have 
\begin{equation*}
r_!(\Delta^n \rightarrow \Delta^0)(v) = (C^n)^{\mathrm{op}},
\end{equation*}
with the notation of the first section and $v$ denoting the unique vertex of $\Delta^0$. Finally, we will need that for each $0 \leq i \leq n+1$, the map $\Delta^{n+1} \rightarrow Q^n$ induces a map
\begin{equation*}
\pi_i: w_!\Delta^{n+1}(0,i) \longrightarrow w_!Q^n(-\infty,v) \simeq (C^n)^{\mathrm{op}}.
\end{equation*}

\section{Review of some model structures}
\label{sec:review}

The relevant model structures in this paper are the \emph{covariant model structure} on the category $\mathbf{sSets}/X$ of simplicial sets over $X$ and the \emph{projective model structure} on the category of simplicial functors $\mathbf{A} \rightarrow \mathbf{sSets}$. We discussed the first of these in the previous paper \cite{leftfibrations}, as well as the second in the case where $\mathbf{A}$ is a discrete category. The case where $\mathbf{A}$ is not necessarily discrete is, in terms of definitions, identical. Indeed, a natural transformation $F \rightarrow G$ between simplicial functors from $\mathbf{A}$ to $\mathbf{sSets}$ is a \emph{projective weak equivalence} (resp. a \emph{projective fibration}) if for each object $a$ of $\mathbf{A}$ the map $F(a) \rightarrow G(a)$ is a weak equivalence (resp. a Kan fibration) of simplicial sets. This model structure is cofibrantly generated, with generating cofibrations and trivial cofibrations of the form
\begin{equation*}
\mathbf{A}(a,-) \times \partial \Delta^n \longrightarrow \mathbf{A}(a,-) \times \Delta^n \quad\quad \text{for } n \geq 0 \text{ and } a \in \mathbf{A}
\end{equation*}
and
\begin{equation*}
\mathbf{A}(a,-) \times \Lambda_k^n \longrightarrow \mathbf{A}(a,-) \times \Delta^n \quad\quad \text{for } 0 \leq k \leq n \text{ and } a \in \mathbf{A}
\end{equation*}
respectively. The projective model structure on $\mathbf{sSets}^{\mathbf{A}}$ is simplicial, where the simplicial structure is the usual one given by
\begin{eqnarray*}
(F \otimes M)(a) & = & F(a) \times M \\
\mathrm{Map}(F, G)_n & = & \mathrm{Hom}(F \otimes \Delta^n, G)
\end{eqnarray*}
for simplicial functors $F, G: \mathbf{A} \rightarrow \mathbf{sSets}$ and a simplicial set $M$.

Furthermore, the adjoint pair $(w_!,w^*)$ constructed in the previous section gives a Quillen equivalence between the Joyal model structure on the category of simplicial sets and the Bergner model structure on the category of simplicial categories. Lurie proves this in \cite{htt}, but his proof depends on Theorem \ref{thm:main}. An independent proof is given by Joyal in \cite{joyalcomparison}, using results from \cite{bergner} and \cite{joyaltierney}. A proof in the more general setting of dendroidal sets and simplicial operads is given in \cite{cisinskimoerdijk3}, Corollary 8.16.

In particular, if $\mathbf{A}$ is a fibrant simplicial category (meaning all its mapping objects are Kan complexes), the counit map
\begin{equation*}
\varepsilon: w_! w^*(\mathbf{A}) \longrightarrow \mathbf{A}
\end{equation*}
is an equivalence of simplicial categories. Also, for a simplicial set $X$ and $w_!X \rightarrow \mathbf{A}$ a fibrant replacement, the map
\begin{equation*}
\eta: X \longrightarrow w^*(\mathbf{A})
\end{equation*}
induced by the unit is a categorical equivalence (i.e. an equivalence in the Joyal model structure). We will need these two facts in combination with the following two invariance results:

\begin{proposition}
\label{prop:invcovariant}
For a categorical equivalence of simplicial sets $f: X \rightarrow Y$, the corresponding adjoint pair
\[
\xymatrix{
f_!: \mathbf{sSets}/X \ar@<.5ex>[r] & \mathbf{sSets}/Y: f^*, \ar@<.5ex>[l]
}
\]
given by composing with and pulling back along $f$, is a Quillen equivalence for the covariant model structures.
\end{proposition}

\begin{proposition}
\label{prop:invprojective}
For an equivalence of simplicial categories $F: \mathbf{A} \rightarrow \mathbf{B}$ (in the Bergner model structure), the corresponding adjoint pair
\[
\xymatrix{
F_!: \mathbf{sSets}^{\mathbf{A}} \ar@<.5ex>[r] & \mathbf{sSets}^{\mathbf{B}}: F^* \ar@<.5ex>[l],
}
\]
is a Quillen equivalence for the projective model structures.
\end{proposition}

The first proposition was already proved by Lurie in \cite{htt}, but his proof depends on the fact that $(r_!, r^*)$ is a Quillen equivalence, as in Theorem \ref{thm:main}. We provided an independent proof in part I of this paper \cite{leftfibrations}. The second proposition can be found in Section A.3.3 of \cite{htt}.

\section{Two Quillen pairs}
\label{sec:Qpairs}

In this section we discuss the adjoint pairs $(h_!, h^*)$ and $(r_!, r^*)$ and prove Propositions \ref{prop:h!} and \ref{prop:r!}. For $\mathbf{A}$ a simplicial category, we start by defining a simplicial functor
\begin{equation*}
h: \mathbf{A}^{\mathrm{op}} \longrightarrow \mathbf{sSets}/w^*\mathbf{A}.
\end{equation*}
For $x$ an object of $\mathbf{A}$, the value $h(x)$ is the map $x/w^*\mathbf{A} \rightarrow w^*\mathbf{A}$. Here, the simplicial set $x/w^*\mathbf{A}$ has as its $n$-simplices the $(n+1)$-simplices of $w^*\mathbf{A}$ whose 0'th vertex is $x$, with the obvious simplicial structure maps obtained from regarding the $(n+1)$-simplex as the join $\Delta^0 \star \Delta^n$. The map from this simplicial set to $w^*\mathbf{A}$ itself is given by $d_0$, i.e. the face map that omits the initial vertex. To make $h$ a simplicial functor, we should provide appropriate maps of simplicial sets $h_{x,y}$ as in the following diagram:
\[
\xymatrix{
\mathbf{A}(y,x) \times x/w^*\mathbf{A} \ar[rr]^{h_{x,y}} \ar[dr] && y/w^*\mathbf{A} \ar[dl] \\
& w^*\mathbf{A}. &
}
\]
Consider an $n$-simplex $\xi$ of the top left simplicial set, i.e. a map
\begin{equation*}
\xi = (\xi_1, \xi_2): \Delta^n \longrightarrow \mathbf{A}(y,x) \times x/w^*\mathbf{A}.
\end{equation*}
Then $\xi_2$ corresponds to a map $\Delta^{n+1} \rightarrow w^*\mathbf{A}$ which sends the 0'th vertex to $x$ and, by adjunction, to a simplicial functor
\begin{equation*}
\overline{\xi}: w_!\Delta^{n+1} \longrightarrow \mathbf{A}.
\end{equation*}
We should define an $n$-simplex $h_{x,y}(\xi): \Delta^n \rightarrow y/w^*\mathbf{A}$, which corresponds to another simplicial functor
\begin{equation*}
H_{x,y}(\xi): w_!\Delta^{n+1} \longrightarrow \mathbf{A}.
\end{equation*}
On objects, it is given by
\begin{equation*}
H_{x,y}(\xi)(i) = \begin{cases}
y & \text{if } i=0 \\
\overline{\xi}(i) & \text{if } 1 \leq i \leq n+1.
\end{cases}
\end{equation*}
If $1\leq i \leq j \leq n+1$, we define 
\begin{equation*}
H_{x,y}(\xi)_{i,j}: w_!\Delta^{n+1}(i,j) \longrightarrow \mathbf{A}(H_{x,y}(\xi)(i), H_{x,y}(\xi)(j))
\end{equation*}
to be the same as 
\begin{equation*}
\overline{\xi}_{i,j}: w_!\Delta^{n+1}(i,j) \longrightarrow \mathbf{A}(\overline{\xi}(i), \overline{\xi}(j)).
\end{equation*}
In the case where $i = 0$ and $j>0$ we need to do more; here we take the composition of the maps
\begin{eqnarray*}
w_!\Delta^{n+1}(0,j) & \longrightarrow & w_!\Delta^{n+1}(0,j) \times w_!\Delta^{n+1}(0,j) \\
& \xrightarrow{(\pi_j,\,\overline{\xi}_{0,j})} & (C^n)^{\mathrm{op}} \times \mathbf{A}(x, \overline{\xi}(j)) \\
& \xrightarrow{(\mu, \,\mathrm{id})} & \Delta^n \times \mathbf{A}(x, \overline{\xi}(j)) \\
& \xrightarrow{(\xi_1, \,\mathrm{id})} & \mathbf{A}(y,x) \times \mathbf{A}(x, \overline{\xi}(j)) \\
& \longrightarrow & \mathbf{A}(y, \overline{\xi}(j)).
\end{eqnarray*}
Here the first map is the diagonal, while the maps $\pi_j$ and $\mu$ have been defined in Section \ref{sec:cubes}. The last map is composition in the simplicial category $\mathbf{A}$. One verifies that the maps defined in this way make $H_{x,y}(\xi)$ into a simplicial functor.

Now that we have the simplicial functor $h: \mathbf{A}^{\mathrm{op}} \rightarrow \mathbf{sSets}/w^*\mathbf{A}$, constructing the functor
\begin{equation*}
h_!: \mathbf{sSets}^{\mathbf{A}} \longrightarrow \mathbf{sSets}/w^*\mathbf{A}
\end{equation*}
is a purely formal matter; indeed, it may be defined as the \emph{simplicial} left Kan extension of $h$ along the simplicial Yoneda embedding
\begin{equation*}
y: \mathbf{A}^{\mathrm{op}} \longrightarrow \mathbf{sSets}^{\mathbf{A}}.
\end{equation*}
Concretely, for a simplicial functor $F: \mathbf{A} \rightarrow \mathbf{sSets}$, we have
\begin{equation*}
h_! F = h \otimes_{\mathbf{A}} F.
\end{equation*}
This tensor product is formed by regarding $F$ as a left module and $h$ as a right module over $\mathbf{A}$ and may be computed as the coequalizer of the following diagram of simplicial sets over $w^*\mathbf{A}$:
\[
\xymatrix{
\coprod_{a,b \in \mathbf{A}} h(b) \otimes (\mathbf{A}(a,b) \times F(a)) \ar@<.5ex>[r] \ar@<-.5ex>[r] & \coprod_{a \in \mathbf{A}} h(a) \otimes F(a).
}
\]
Here the tensor symbol refers to the tensoring of the category $\mathbf{sSets}/w^*\mathbf{A}$ over the category $\mathbf{sSets}$ of simplicial sets, as explained in the previous section. The two arrows arise from the left module structure of $F$ and the right module structure of $h$ respectively. From this formula it is immediately clear that $h_!$ itself is a simplicial functor, in the sense that there are natural isomorphisms
\begin{equation*}
h_!(F \otimes M) \simeq h_!(F) \otimes M
\end{equation*}
for any simplicial functor $F: \mathbf{A} \rightarrow \mathbf{sSets}$ and simplicial set $M$. Also, by formal reasoning, the functor $h_!$ admits a right adjoint $h^*$.

\begin{proof}[Proof of Proposition \ref{prop:h!}]
To show that $h_!$ is left Quillen, we should check that it sends the generating cofibrations
\begin{equation*}
\mathbf{A}(a,-) \times \partial \Delta^n \longrightarrow \mathbf{A}(a,-) \times \Delta^n \quad\quad \text{for } n \geq 0 \text{ and } a \in \mathbf{A}
\end{equation*}
and generating trivial cofibrations
\begin{equation*}
\mathbf{A}(a,-) \times \Lambda_k^n \longrightarrow \mathbf{A}(a,-) \times \Delta^n \quad\quad \text{for } 0 \leq k \leq n \text{ and } a \in \mathbf{A}
\end{equation*}
to cofibrations (resp. trivial cofibrations) in the covariant model structure on $\mathbf{sSets}/w^*\mathbf{A}$. We just observed that $h_!$ is simplicial and sends these to the maps
\begin{equation*}
a/w^*\mathbf{A} \otimes \partial \Delta^n \longrightarrow a/w^*\mathbf{A} \otimes \Delta^n
\end{equation*}
and
\begin{equation*}
a/w^*\mathbf{A} \otimes \Lambda_k^n \longrightarrow a/w^*\mathbf{A} \otimes \Delta^n
\end{equation*}
respectively. These are indeed cofibrations (resp. trivial cofibrations), simply because the covariant model structure is simplicial. To verify the claim about homotopy colimits, observe that the composition
\begin{equation*}
\mathbf{sSets}^{\mathbf{A}} \xrightarrow{h_!} \mathbf{sSets}/w^*\mathbf{A} \rightarrow \mathbf{sSets}
\end{equation*}
preserves homotopy colimits (since both functors are left Quillen) and furthermore sends every representable functor $\mathbf{A}(a,-)$ to a weakly contractible simplicial set (namely $a/w^*\mathbf{A}$). But these two facts characterize the homotopy colimit functor up to weak equivalence.
\end{proof}

Before proving Proposition \ref{prop:r!}, we observe the following naturality properties of the functor $r_!$. First of all, it follows straightforwardly from the definitions that for a map of simplicial sets $f: Y \rightarrow Z$ the square
\[
\xymatrix{
\mathbf{sSets}/Y \ar[r]^{r_!}\ar[d]_{f_!} & \mathbf{sSets}^{w_! Y} \ar[d]^{(w_! f)_!} \\
\mathbf{sSets}/Z \ar[r]_{r_!} & \mathbf{sSets}^{w_! Z}
}
\]
commutes up to natural isomorphism. Here the left vertical functor is given by composition with $f$, whereas the right vertical functor is the pushforward along the simplicial functor $w_!f$. The horizontal functors are both given by $r_!$, the top one applied to simplicial sets over $Y$, the bottom one to simplicial sets over $Z$. In particular, for a commutative square of simplicial sets
\[
\xymatrix{
U \ar[r]\ar[d]_p & V \ar[d]^q \\
Y \ar[r]_f & Z
}
\]
we obtain a map
\begin{equation*}
r_!(p) \longrightarrow (w_!f)^*r_!(q)
\end{equation*}
of simplicial diagrams on $w_!Y$.

\begin{proof}[Proof of Proposition \ref{prop:r!}]
To check that $r_!$ is left Quillen, it will suffice to check that it sends maps in $\mathbf{sSets}/X$ of the form
\[
\xymatrix{
\partial \Delta^n \ar[rr]\ar[dr] && \Delta^n \ar[dl]^p \\
& X &
}
\]
to cofibrations and, for $0 \leq k < n$, maps of the form
\[
\xymatrix{
\Lambda_k^n \ar[rr]\ar[dr] && \Delta^n \ar[dl]^p \\
& X &
}
\]
to trivial cofibrations (see \cite{leftfibrations}). By the naturality of $r_!$ discussed above and the fact that $(w_!p)_!$ is left Quillen, we may reduce our verifications to the case where $X = \Delta^n$.

Let us first consider the effect of $r_!: \mathbf{sSets}/\Delta^n \rightarrow \mathbf{sSets}^{w_!\Delta^n}$ on the inclusion $i: \partial \Delta^n \rightarrow \Delta^n$, where domain and codomain are considered as simplicial sets over $\Delta^n$ in the obvious way. For each $m < n$ the map
\begin{equation*}
r_!(\partial \Delta^n)(m) \rightarrow r_!(\Delta^n)(m)
\end{equation*}
is an isomorphism.  At the final vertex, the map
\begin{equation*}
r_!(\partial \Delta^n)(n) \rightarrow r_!(\Delta^n)(n) \simeq \prod_{0 \leq m < n} I
\end{equation*}
is the inclusion of the union of the subcubes
\begin{equation*}
\prod_{0 \leq m < k} I \times \{-\} \times \prod_{k < m < n} I \quad \text{for } 0 \leq k \leq n-1,
\end{equation*}
where the $k$'th cube is contributed by $r_!(d_k\Delta^n)(n)$, and
\begin{equation*}
\prod_{0 \leq m < k} I \times \{+\} \times \prod_{k < m < n} I \quad \text{for } 0 \leq k \leq n-1,
\end{equation*}
which are all contained in $r_!(d_n\Delta^n)(n)$. In other words, this map is the inclusion of all the faces of the cube $r_!(\Delta^n)(n)$. In particular, this map is a cofibration of simplicial sets, for which we write $j$. Then $r_!(i)$, being an isomorphism at each object of $w_!\Delta^n$ except the last, is a pushout of the map $w_!\Delta^n(n,-) \otimes j$ and hence a projective cofibration.

Similarly, consider the horn inclusion $i: \Lambda_k^n \rightarrow \Delta^n$ for $0 \leq k < n$. Again, for each $m < n$, the map
\begin{equation*}
r_!(\Lambda_k^n)(m) \rightarrow r_!(\Delta^n)(m)
\end{equation*}
is an isomorphism. At the final vertex,
\begin{equation*}
r_!(\Lambda_k^n)(n) \rightarrow r_!(\Delta^n)(n) \simeq \prod_{0 \leq m < n} I
\end{equation*}
is now the inclusion of the union of the subcubes
\begin{equation*}
\prod_{0 \leq m < l} I \times \{-\} \times \prod_{l < m < n} I \quad \text{for } 0 \leq l \leq n-1 \text{ and } l \neq k
\end{equation*}
and
\begin{equation*}
\prod_{0 \leq m < l} I \times \{+\} \times \prod_{l < m < n} I \quad \text{for } 0 \leq l \leq n-1.
\end{equation*}
Write $j: K \rightarrow \prod_{0 \leq m < n} I$ for this inclusion. Again, $r_!(i)$ is a pushout of the map $w_!\Delta^n(n,-) \otimes j$, so we should show that $j$ is a trivial cofibration of simplicial sets. It is clearly a cofibration; we will conclude the proof by showing that $K$ is weakly contractible. To see this, note that the projection
\begin{equation*}
\prod_{0 \leq m < n} I \longrightarrow \prod_{0 \leq m < k} I \times \{+\} \times \prod_{k < m < n} I 
\end{equation*}
is part of an evident deformation retraction, which restricts to a deformation retraction of $K$ onto $ \prod_{0 \leq m < k} I \times \{+\} \times \prod_{k < m < n} I$. In particular, $K$ is indeed weakly contractible.
\end{proof}

%We finish this section by justifying Remark \ref{rmk:tilder} from the introduction. For the moment, adapt the notation of that remark. Consider the functor 
%\begin{equation*}
%\Delta^-: [n] \rightarrow \mathbf{sSets}: i \longmapsto \Delta^i,
%\end{equation*}
%where for $i<j$ the corresponding map $\Delta^i \rightarrow \Delta^j$ is the inclusion of the initial segment $\Delta^{\{0, \ldots, i\}}$. To bring out the analogy between $\tilde{r}_!$ and $r_!$, we note that for an ordinary category $\mathbf{A}$ the functor 
%\begin{equation*}
%r_!: \mathbf{sSets}/N\mathbf{A} \longrightarrow \mathbf{sSets}^{\mathbf{A}}
%\end{equation*}
%defined in part I of this paper \cite{leftfibrations} may be described by the following formula:
%\begin{equation*}
%r_!(\Delta^n \xrightarrow{p} X)(x) = (\tau p)^* \mathbf{A}(-,x) \otimes_{[n]} \Delta^-.
%\end{equation*}
%Here $\tau$ is the left adjoint of the nerve functor $N: \mathbf{Cat} \rightarrow \mathbf{sSets}$. Using the map $\mu \circ \pi_i: w_!\Delta^n(-\infty, i)$

\section{Proof of Proposition \ref{prop:h!leftinv}}

We will prove Proposition \ref{prop:h!leftinv} by first giving a natural transformation $\nu$ from the functor
\begin{equation*}
r_! \circ h_!: \mathbf{sSets}^{\mathbf{A}} \longrightarrow \mathbf{sSets}^{w_!w^*\mathbf{A}}
\end{equation*}
to the functor
\begin{equation*}
\varepsilon^*: \mathbf{sSets}^{\mathbf{A}} \longrightarrow \mathbf{sSets}^{w_!w^*\mathbf{A}}.
\end{equation*}
Subsequently we prove that the adjoint natural transformation
\begin{equation*}
\bar{\nu}: \varepsilon_! \circ r_! \circ h_! \longrightarrow \mathrm{id}
\end{equation*}
is a weak equivalence when evaluated on projectively cofibrant objects. It follows that the derived functor $\mathbf{L}r_! \circ \mathbf{L}h_!$ is a right quasi-inverse to $\mathbf{L}\varepsilon_!$; the latter functor is an equivalence of categories with quasi-inverse $\mathbf{R}\varepsilon^*$, so it follows that the derived functors $\mathbf{L}h_! \circ \mathbf{L}r_!$ and $\mathbf{R}\varepsilon^*$ must be naturally isomorphic.

Since $r_! \circ h_!$ preserves colimits and every object of $\mathbf{sSets}^\mathbf{A}$ can canonically be expressed as a colimit of representable functors of the form $\mathbf{A}(a,-) \times \Delta^n$, it suffices to define $\nu$ on these. Let us first consider the case $n=0$. By definition,
\begin{equation*}
h_! \mathbf{A}(a,-) = (p_a: a/w^*\mathbf{A} \rightarrow w^*\mathbf{A}).
\end{equation*}
The map of simplicial sets $p_a$ in turn is a colimit over the simplices $\xi: \Delta^n \rightarrow a/w^*\mathbf{A}$ of the maps $p_a \xi: \Delta^n \rightarrow w^*\mathbf{A}$. Therefore, for an object $b \in \mathbf{A}$, the simplicial set $r_!p_a(b)$ is a colimit over those simplices $\xi$ of the following simplicial sets:
\begin{equation*}
r_!(p_a \xi)(b) = w_!(p_a \xi)^*\bigl(w_!w^*\mathbf{A}(-,b)\bigr) \otimes_{w_!\Delta^n} w_!(\Delta^{n\,\triangleleft})(\infty,-).
\end{equation*}
The map $\xi$ determines by adjunction a simplicial functor
\begin{equation*}
\bar{\xi}: w_!(\Delta^{n\,\triangleleft}) \longrightarrow \mathbf{A} 
\end{equation*}
which sends the initial vertex $-\infty$ to $a$. Using this functor together with the isomorphism
\begin{equation*}
r_!(p_a \xi)(b) \simeq w_!w^*\mathbf{A}(-,b) \otimes_{w_!w^*\mathbf{A}} w_!(p_a\xi)_!\bigl(w_!(\Delta^{n\,\triangleleft})(-\infty,-)\bigr)
\end{equation*}
and subsequently applying $\varepsilon_!$ we find natural maps
\begin{eqnarray*}
r_!(p_a \xi)(b) & \longrightarrow & \mathbf{A}(-,b) \otimes_{\mathbf{A}} \mathbf{A}(a,-) \\
& \simeq & \mathbf{A}(a,b).
\end{eqnarray*}
Passing to colimits, we have thus defined the required natural transformation
\begin{equation*}
\nu(a): r_!h_!\mathbf{A}(a,-) \longrightarrow \varepsilon^*\mathbf{A}(a,-).
\end{equation*}
To treat the general case
\begin{equation*}
\nu(a,n): r_!h_!\bigl(\mathbf{A}(a,-) \otimes \Delta^n \bigr) \longrightarrow \varepsilon^*\mathbf{A}(a,-) \otimes \Delta^n,
\end{equation*}
note first that, because $h_!$ is compatible with the simplicial structure, we have 
\begin{equation*}
r_!h_!\bigl(\mathbf{A}(a,-) \otimes \Delta^n \bigr) \simeq r_!\bigl(p_a \otimes \Delta^n \bigr).
\end{equation*}
For a moment let us carry along the `base' in the notation for $r_!$, i.e. write
\begin{equation*}
r_!^{w^*\mathbf{A}}: \mathbf{sSets}/w^*\mathbf{A} \rightarrow \mathbf{sSets}^{w_!w^*\mathbf{A}} \quad\quad \text{and} \quad\quad r_!^{\Delta^0}: \mathbf{sSets} \rightarrow \mathbf{sSets}. 
\end{equation*}
Then by the naturality properties of $r_!$ discussed just before the proof of Proposition \ref{prop:r!}, the projections onto the two factors give natural maps
\begin{eqnarray*}
r^{w^*\mathbf{A}}_!\bigl(p_a \otimes \Delta^n \bigr) & \longrightarrow & r_!^{w^*\mathbf{A}}(p_a) \otimes r_!^{\Delta^0}(\Delta^n) \\ 
& \xrightarrow{(\nu(a),\mu)} &  \varepsilon^*\mathbf{A}(a,-)\otimes \Delta^n, 
\end{eqnarray*}
where $\nu(a)$ is as defined above and $\mu$ is defined as in Section \ref{sec:cubes}. This completes the definition of $\nu$.

\begin{proof}[Proof of Proposition \ref{prop:h!leftinv}]
As explained at the beginning of this section, we should prove that $\bar{\nu}: \varepsilon_! r_! h_! \rightarrow \mathrm{id}$ is a weak equivalence when evaluated on projectively cofibrant objects. Since all functors involved are left Quillen, it follows by the usual arguments on skeletal filtrations of cofibrant objects that it suffices to prove this statement on representables of the form $\mathbf{A}(a,-) \otimes \Delta^n$. Consider the inclusion $0: \Delta^0 \rightarrow \Delta^n$ and the induced square
\[
\xymatrix{
\varepsilon_! r_! h_!(\mathbf{A}(a,-) \otimes \Delta^0) \ar[d]\ar[r]^-{\bar{\nu}(a,0)} & \mathbf{A}(a,-) \otimes \Delta^0 \ar[d] \\
\varepsilon_! r_! h_!(\mathbf{A}(a,-) \otimes \Delta^n) \ar[r]_-{\bar{\nu}(a,n)} & \mathbf{A}(a,-) \otimes \Delta^n.
}
\]
The right vertical map is a trivial cofibration (since the projective model structure is simplicial), so the left vertical map is a trivial cofibration as well (since $\varepsilon_! r_! h_!$ is left Quillen). Therefore, to prove that $\bar{\nu}(a,n)$ is a weak equivalence, it suffices to prove that $\bar{\nu}(a,0)$ is a weak equivalence. Consider the following maps of simplicial sets:
\[
\xymatrix{
\Delta^0 \ar[dr]_a \ar[rr] && a/w^*\mathbf{A} \ar[dl] \\
& w^*\mathbf{A} &
}
\]
According to Lemma 2.4 of \cite{leftfibrations}, the horizontal map is a weak equivalence in the covariant model structure over $w^*\mathbf{A}$. Now consider the following commutative diagram:
\[
\xymatrix{
\varepsilon_! r_!(\Delta^0 \xrightarrow{a} w^*\mathbf{A}) \ar[r]\ar[dr] & \varepsilon_! r_!(a/w^*\mathbf{A} \rightarrow w^*\mathbf{A}) \ar[d]^{\bar{\nu}(a,0)} \\
& \mathbf{A}(a,-).
}
\]
The horizontal arrow is a weak equivalence in the projective model structure since $\varepsilon_!r_!$ is left Quillen; furthermore, the reader will easily verify that the slanted arrow is an isomorphism. By two-out-of-three, $\bar{\nu}(a,0)$ is a weak equivalence.
\end{proof}

\section{Proof of Proposition \ref{prop:h!rightinv}}

Much as in Part I of this paper, we will prove Proposition \ref{prop:h!rightinv} by exhibiting a zigzag of natural transformations
\begin{equation*}
h_!r_! \xrightarrow{\gamma} L \xleftarrow{\iota} \eta_!
\end{equation*}
between functors from $\mathbf{sSets}/X$ to $\mathbf{sSets}/w^*w_!X$ and subsequently proving that both are covariant weak equivalences over $w^*w_!X$. 

We will define the natural transformation $\gamma: h_!r_! \rightarrow L$ as a composition of natural transformations
\begin{equation*}
h_!r_! \xrightarrow{\zeta} Z \xrightarrow{\zeta'} L.
\end{equation*}
Informally speaking, a simplex of $L$ has the shape of a ladder in $w^*w_!X$, whereas a simplex of $Z$ traces out a `$Z$-shape' in that ladder. To understand the origin of the definitions we give in this section, the reader might find it helpful to compare them to the definitions we made in Part I.

To describe the relevant functors, first define a simplicial category $S^n$ by the following pushout square:
\[
\xymatrix{
w_!\Delta^0 \ar[r]^{n}\ar[d]_{0} & w_!\Delta^n \ar[d] \\
w_!\Delta^{n+1} \ar[r] & S^n.
}
\]
Reading $\Delta^{n+1}$ as $\Delta^0 \star \Delta^n$, the definition of $S^n$ is natural in $[n]$ in an evident manner, so that $S^\bullet$ is a cosimplicial object in simplicial categories. Note also that there is an inclusion of simplicial categories
\begin{equation*}
S^n \longrightarrow w_!\Delta^{2n+1},
\end{equation*}
coming from the two inclusions of the `segments' $\Delta^n \simeq \Delta^{\{0, \ldots, n\}} \subseteq \Delta^{2n+1}$ and $\Delta^{n+1} \simeq \Delta^{\{n, \ldots, 2n+1\}} \subseteq \Delta^{2n+1}$. Informally speaking, the difference between $S^n$ and $w_!\Delta^{2n+1}$ is that there is no `time' $t_n$ assigned to the object $n$ of $S_n$; the inclusion assigns to this vertex the coordinate $+$.

We need another simplicial category $T^n$, defined by the following pullback square:
\[
\xymatrix{
T^n \ar[r]\ar[d] & S^n \ar[d] \\
w_!(\Delta^1 \times \Delta^n) \ar[r] & w_!\Delta^{2n+1}.
}
\]
Here the bottom horizontal functor comes from the inclusion $\Delta^1 \times \Delta^n \rightarrow \Delta^{2n+1}$ which sends the vertex $(i,j)$ to $(n+1)i + j$. Again, $T^\bullet$ is a cosimplicial object in simplicial categories in an evident way. For a map of simplicial sets $p: Y \rightarrow X$, we can now define the simplicial sets $Z(p)$ and $L(p)$ by the following pullback squares:
\[
\xymatrix{
Z(p)_n \ar[d]\ar[r] & \mathrm{Hom}(S^n, w_! X) \ar[d] & L(p)_n \ar[d]\ar[r] & \mathrm{Hom}(T^n, w_! X) \ar[d] \\
Y_n \ar[r] & \mathrm{Hom}(w_!\Delta^n, w_!X), & Y_n \ar[r] & \mathrm{Hom}(w_!\Delta^n, w_! X).
}
\]
In these squares, the right vertical maps are induced by the inclusions $w_!\Delta^n \simeq w_!\Delta^{\{0, \ldots, n\}} \rightarrow S^n$ and $w_!\Delta^n \simeq w_!(\{0\} \times \Delta^n) \rightarrow T^n$. Thus, one can think of the $n$-simplices of $Z(p)$ (resp. $L(p)$) as corresponding to simplicial functors $S^n \rightarrow w_!X$ (resp. $T^n \rightarrow w_!X$) with a lift of the initial segment $\Delta^{\{0, \ldots, n\}}$ (resp. $\{0\} \times \Delta^n)$) to $Y$.

Of course there are also functors
\begin{equation*}
w_!\Delta^n \simeq w_!(\{1\} \times \Delta^n) \longrightarrow T^n \longrightarrow S^n.
\end{equation*}
These induce maps
\begin{equation*}
\mathrm{Hom}(S^n, w_! X) \longrightarrow \mathrm{Hom}(T^n, w_! X) \longrightarrow \mathrm{Hom}(w_!\Delta^n, w_!X) = (w^*w_!X)_n
\end{equation*}
and consequently we find maps
\begin{equation*}
Z(p) \xrightarrow{\zeta'} L(p) \rightarrow w^*w_!X,
\end{equation*}
also giving $Z(p)$ and $L(p)$ the structure of simplicial sets over $w^*w_!X$. We should still define the maps $\iota$ and $\zeta$ in the following diagram:
\[
\xymatrix{
h_!r_!(p) \ar[drr]\ar[r]^\zeta & Z(p) \ar[r]^{\zeta'} & L(p) \ar[d] & \eta_!(p)\ar[l]_{\iota} \ar[dl] \\
&& w^*w_!X. &
}
\]
For $\iota$, we use the functor $T^n \rightarrow w_!\Delta^n$ induced by the projection $\Delta^1 \times \Delta^n \rightarrow \Delta^n$. Indeed, this map allows us to form the composition
\begin{equation*}
Y_n \longrightarrow \mathrm{Hom}(w_!\Delta^n, w_!X) \longrightarrow \mathrm{Hom}(T^n, w_!X)
\end{equation*}
which in turn gives the required $\iota$. Indeed, from the pullback square defining $L(p)$, we find a map of simplicial sets $Y \rightarrow L(p)$, and it is easily verified that it is compatible with the maps of both these simplicial sets to $w^*w_!X$. 

Finally, to define $\zeta$ we proceed as follows. Since $h_!r_!$ preserves colimits, it suffices to treat the case where $Y$ is a simplex $\Delta^n$, so that $p$ is a map $\Delta^n \rightarrow X$. Recall that 
\begin{equation*}
r_!(p)(x) = (w_!p)^*\bigl(w_!X(-,x)\bigr) \otimes_{w_!\Delta^n} w_!(\Delta^{n\,\triangleleft})(-\infty,-).
\end{equation*}
Write $R$ for the composition of simplicial functors
\begin{equation*}
(w_!\Delta^n)^{\mathrm{op}} \xrightarrow{w_!p} (w_!X)^{\mathrm{op}} \xrightarrow{h} \mathbf{sSets}/w^*w_!X.
\end{equation*}
In particular, note that
\begin{equation*}
R(i) = \bigl(p(i)/w^*w_!X \longrightarrow w^*w_!X \bigr).
\end{equation*}
Since $h_!$ is simplicial and commutes with colimits, it also commutes with tensoring with left modules over a simplicial category. We therefore have
\begin{equation*}
h_!r_!(p) = R(-) \otimes_{w_!\Delta^n} w_!(\Delta^{n\,\triangleleft})(-\infty, -).
\end{equation*}
Recall from Section \ref{sec:Qpairs} the definition of this tensor product as a coequalizer; in particular, $h_!r_!(p)$ is a quotient of the simplicial set
\begin{equation*}
P := \coprod_{0 \leq i \leq n} p(i)/w^*w_!X \otimes w_!(\Delta^{n\,\triangleleft})(-\infty, i).
\end{equation*}
As in Section \ref{sec:cubes}, there are maps $w_!(\Delta^n)^{\triangleleft}(-\infty,i) \xrightarrow{\pi_i} (C^n)^{\mathrm{op}} \xrightarrow{\mu} \Delta^n$. Furthermore, this composition clearly factors through the initial segment $\Delta^i \simeq \Delta^{\{0, \ldots, i\}} \subseteq \Delta^n$. Thus we obtain a map
\begin{equation*}
P \longrightarrow \coprod_{0 \leq i \leq n} p(i)/w^*w_!X \otimes \Delta^i.
\end{equation*}
Consider a $k$-simplex $\xi = (\xi_1, \xi_2)$ of this simplicial set and assume that $\xi_2: \Delta^k \rightarrow \Delta^i$ is surjective (without loss of generality; any other simplex is a face of such). Then $\xi_1$ gives a simplicial functor $w_!\Delta^{k+1} \rightarrow w_!X$ sending $0$ to $p(i)$ and from $\xi_2$ we may form a composition of simplicial functors
\begin{equation*}
w_!\Delta^k \xrightarrow{w_!\xi_2} w_!\Delta^i \xrightarrow{w_!p} w_!X
\end{equation*}
which (by surjectivity of $\xi_2$) sends the object $k$ of $w_!\Delta^k$ to $p(i)$. To apply $w_!p$ we have again regarded $\Delta^i$ as an initial segment of $\Delta^n$. These two functors may then be amalgamated to obtain a simplicial functor
\begin{equation*}
w_!\Delta^k \amalg_{w_!\Delta^0} w_!\Delta^{k+1} = S^k \longrightarrow w_!X.
\end{equation*}
This procedure defines a map of simplicial sets
\begin{equation*}
P \longrightarrow Z(p)
\end{equation*}
and one readily verifies that it descends to the quotient of $P$ defining $h_!r_!(p)$, thus giving the desired map $\zeta$.

Before proving Proposition \ref{prop:h!rightinv} let us note some formal properties of our constructions. Also recall from part I of this paper the notion of a \emph{covariant deformation retract}, which we will use below.

\begin{lemma}
\label{lem:LZformal}
The functors $L$ and $Z$ preserve colimits and cofibrations.
\end{lemma}
\begin{proof}
Preservation of colimits is a direct consequence of the fact that pullbacks in simplicial sets preserve colimits. Similarly, pullbacks preserve monomorphisms, proving the claim about cofibrations. 
\end{proof}

\begin{proof}[Proof of Proposition \ref{prop:h!rightinv}]
Consider a map of simplicial sets $p: Y \rightarrow X$. We will prove that the maps
\begin{equation*}
h_!r_!(p) \xrightarrow{\gamma} L(p) \xleftarrow{\iota} \eta_!(p)
\end{equation*}
are both covariant weak equivalences over $w^*w_!X$. First we prove that $\iota$ is a covariant deformation retract. The retraction $r: L(p) \rightarrow \eta_!(p)$ is defined by precomposing functors $T^n \rightarrow w_!X$ with the functor
\begin{equation*}
w_!\Delta^n \simeq w_!(\{0\} \times \Delta^n) \longrightarrow T^n.
\end{equation*}
This functor is a section of the functor $T^n \rightarrow w_!\Delta^n$ used to define $\iota$, so that $r\iota = \mathrm{id}$ as required. It remains to provide a map
\begin{equation*}
k: L(p) \times \Delta^1 \longrightarrow L(p)
\end{equation*}
that exhibits a homotopy from $\iota r$ to $\mathrm{id}_{L(p)}$. Let $\xi = (\xi_1,\xi_2): \Delta^n \rightarrow L(p) \times \Delta^1$ be an $n$-simplex of the domain of $k$. Then $\xi_2$ defines a self-map $\lambda_{\xi_2}: \Delta^1 \times \Delta^n \rightarrow \Delta^1 \times \Delta^n$, which is determined by its effect on vertices as follows:
\begin{equation*}
\lambda_{\xi_2}(i,j) = \begin{cases}
(i,j) & \text{if } \xi_2(j) = 1, \\
(0,j) & \text{otherwise.}
\end{cases}
\end{equation*}
Note that if $\xi_2$ is the constant function with value 0, then $\lambda_{\xi_2}$ is the projection $\Delta^1 \times \Delta^n \rightarrow \{0\} \times \Delta^n$, whereas if $\xi_2$ is constant with value $1$ then $\lambda_{\xi_2}$ is the identity. Also, $\lambda_{\xi_2}$ induces a simplicial functor
\begin{equation*}
w_!\lambda_{\xi_2}: w_!(\Delta^1 \times \Delta^n) \longrightarrow w_!(\Delta^1 \times \Delta^n)
\end{equation*}
which is easily seen to restrict to a simplicial functor $T^n \rightarrow T^n$, for which we write $\Lambda_{\xi_2}$. The $n$-simplex $\xi_1$ corresponds to a functor $\overline{\xi}_1: T^n \rightarrow w_!X$. We now define the $n$-simplex $k(\xi): \Delta^n \rightarrow L(p)$ to be the one corresponding to the composition of functors $\overline{\xi}_1 \circ \Lambda_{\xi_2}: T^n \rightarrow w_!X$. It is easily verified that the map $k$ defined in this way is indeed a homotopy from $\iota r$ to the identity of $L(p)$.

We will now verify that $h_!r_!(p) \xrightarrow{\gamma} L(p)$ is a covariant weak equivalence. The functors $h_!r_!$ and $L$ both preserve colimits and cofibrations, so the usual `cube lemma' and an induction on the skeletal filtration of $Y$ (as in the proof of Proposition 5.5 of \cite{leftfibrations}) show that it suffices to treat the case of a simplex $p: \Delta^n \rightarrow X$. The inclusion $\Delta^0 \xrightarrow{0} \Delta^n$ is a covariant weak equivalence, which is preserved by $h_!r_!$ (since it is left Quillen) and by $L$ (which is also left Quillen, since we proved it is weakly equivalent to $\eta_!$). Therefore we may further reduce to the case of a vertex $x: \Delta^0 \rightarrow X$. Note that $\eta_!(x)$ is just the map $x: \Delta^0 \rightarrow w^*w_!X$ corresponding to the object $x$ of the simplicial category $w_!X$. Also observe that $h_!r_!(x) = (x/w^*w_!X \rightarrow w^*w_!X)$. Clearly there is a vertex $\mathrm{id}_x: \Delta^0 \rightarrow x/w^*w_!X$, corresponding to the identity morphism of $x$, which makes the following diagram of simplicial sets over $w^*w_!X$ commute: 
\[
\xymatrix{
\eta_!(x) \ar[d]_{\mathrm{id}_x} \ar[dr]^{\iota} & \\ 
x/w^*w_!X \ar[r]_-{\gamma} & L(x).
}
\]
The left vertical map is a covariant weak equivalence by Lemma 2.4 of \cite{leftfibrations}. We already proved that $\iota$ is a covariant weak equivalence. By two-out-of-three, $\gamma$ is therefore a weak equivalence as well.
\end{proof}

\section{Proof of Theorem \ref{thm:main}}
\label{sec:proof}

Let $\mathbf{A}$ be a fibrant simplicial category. As explained in the first section, Proposition \ref{prop:h!leftinv} proves that $\mathbf{L}h_!$ has a left quasi-inverse. To conclude that $(h_!,h^*)$ is a Quillen equivalence, it thus suffices to show that it has a right quasi-inverse as well. For this we need a lemma expressing the compatibility of $h_!$ with `change of base'. If $F: \mathbf{A} \rightarrow \mathbf{B}$ is a simplicial functor consider the square
\[
\xymatrix{
\mathbf{sSets}^{\mathbf{A}} \ar[d]_{F_!}\ar[r]^-{h_!^{\mathbf{A}}} & \mathbf{sSets}/w^*\mathbf{A} \ar[d]^{(w^*F)_!} \\
\mathbf{sSets}^{\mathbf{B}} \ar[r]_-{h_!^{\mathbf{B}}} & \mathbf{sSets}/w^*\mathbf{B},
}
\]
in which the horizontal functors are $h_!$ applied to simplicial functors on $\mathbf{A}$ and $\mathbf{B}$ respectively. This square, however, generally does \emph{not} commute up to natural isomorphism. Rather, it commutes up to weak equivalence in the following sense:

\begin{lemma}
\label{lem:h!basechange}
There is a natural transformation $H: (w^*F)_! \circ h_!^{\mathbf{A}} \rightarrow h_!^{\mathbf{B}} \circ F_!$, which is a covariant weak equivalence over $w^*\mathbf{B}$ when evaluated on projectively cofibrant objects.
\end{lemma}
\begin{proof}
Since all functors involved preserve colimits and are compatible with the simplicial structure, it suffices to define $H$ on representable functors $\mathbf{A}(a,-)$ for objects $a \in \mathbf{A}$. Note that $(w^*F)_!h_!^{\mathbf{A}}\mathbf{A}(a,-)$ is the map of simplicial sets $a/w^*\mathbf{A} \rightarrow w^*\mathbf{B}$. Recall that the $n$-simplices of $a/w^*\mathbf{A}$ are $(n+1)$-simplices of $w^*\mathbf{A}$ whose first vertex is $a$; clearly the map of simplicial sets $w^*F$ then induces a map of simplicial sets over $w^*\mathbf{B}$ as follows:
\[
\xymatrix{
a/w^*\mathbf{A} \ar[rr]^{H(\mathbf{A}(a,-))} \ar[dr] && F(a)/w^*\mathbf{B} \ar[dl] \\ 
& w^*\mathbf{B}. &
}
\]
It is easily verified that this $H$ defines a natural transformation. To see that it is a covariant weak equivalence, the usual skeletal induction shows that it suffices to check this on functors of the form $\mathbf{A}(a,-) \otimes \Delta^n$ and the fact that the functors involved are compatible with the simplicial structure allows us to reduce further to the case $n=0$. Now consider the following diagram of simplicial sets over $w^*\mathbf{B}$:
\[
\xymatrix{
\Delta^0 \ar[d]_{\mathrm{id}_a}\ar[dr]^{\mathrm{id}_{F(a)}} & \\
a/w^*\mathbf{A} \ar[r] & F(a)/w^*\mathbf{B}.
}
\]
The vertical and slanted maps are covariant weak equivalences by Lemma 2.4 of \cite{leftfibrations}, so that $H(\mathbf{A}(a,-)$ is a covariant weak equivalence as well.
\end{proof}

To deduce the existence of a right quasi-inverse to $h_!$, choose a simplicial set $X$ and an equivalence of simplicial categories $F: w_!X \rightarrow \mathbf{A}$. (For example, one might simply take $X = w^*\mathbf{A}$ and take $F$ to be the counit $\varepsilon$.) Write $\overline{F}: X \rightarrow w^*\mathbf{A}$ for the adjoint map, which is a categorical equivalence of simplicial sets. Now consider the following sequence of natural isomorphisms:
\begin{eqnarray*}
\mathbf{L}h_! \circ \mathbf{L}F_! \circ \mathbf{L}r_! & \simeq & \mathbf{L}(w^*F)_! \circ \mathbf{L}h_! \circ \mathbf{L}r_! \\
& \simeq & \mathbf{L}(w^*F)_! \circ \mathbf{L}\eta_! \\
& \simeq & \mathbf{L}\overline{F}_!.
\end{eqnarray*}
Here the first isomorphism follows from Lemma \ref{lem:h!basechange}, the second from Proposition \ref{prop:h!rightinv} and the third one from the identification $\overline{F} = w^*F \circ \eta$. Since $\overline{F}$ is a categorical equivalence, Proposition \ref{prop:invcovariant} guarantees that $\mathbf{L}\overline{F}_!$ is an equivalence of categories. We conclude that $\mathbf{L}h_!$ admits a right quasi-inverse.

To complete the proof of Theorem \ref{thm:main} we should argue that for a simplicial set $X$ the adjoint pair 
\[
\xymatrix{
r_!: \mathbf{sSets}/X \ar@<.5ex>[r] & \mathbf{sSets}^{w_!X}: r^* \ar@<.5ex>[l]
}
\]
is a Quillen equivalence. Choose a simplicial category $\mathbf{A}$ and a categorical equivalence $\overline{F}: X \rightarrow w^*\mathbf{A}$. If we write $F: w_!X \rightarrow \mathbf{A}$ for the adjoint of this map, the sequence of natural isomorphisms above shows that $\mathbf{L}r_!$ admits a left quasi-inverse, which is itself an equivalence of categories. Thus $\mathbf{L}r_!$ is an equivalence of categories as well.

\bibliographystyle{plain}
\bibliography{biblio}

\end{document}